\documentclass[10pt]{amsart}
\usepackage{latexsym, amsmath,amssymb, hyperref}

\usepackage{amsthm}

\renewcommand{\epsilon}{\varepsilon}
\newcommand{\newsection}[1]
{\subsection{#1}\setcounter{theorem}{0} \setcounter{equation}{0}
\par\noindent}

\newtheorem{theorem}{Theorem}

\newtheorem{lemma}[theorem]{Lemma}
\newtheorem{corr}[theorem]{Corollary}

\newtheorem{proposition}[theorem]{Proposition}
\newtheorem{deff}[theorem]{Definition}

\newcommand{\bth}{\begin{theorem}}
\newcommand{\ble}{\begin{lemma}}
\newcommand{\bcor}{\begin{corr}}

\newcommand{\bdeff}{\begin{deff}}

\newcommand{\bprop}{\begin{proposition}}
\newcommand{\ele}{\end{lemma}}
\newcommand{\ecor}{\end{corr}}
\newcommand{\edeff}{\end{deff}}

\newcommand{\eprop}{\end{proposition}}

\newcommand{\cd}{\, \cdot\, }

\newcommand{\la}{\lambda}

\newcommand{\e}{\varepsilon}

\renewcommand{\Pi}{\varPi}

\renewcommand{\epsilon}{\varepsilon}

\newcommand{\R}{{\mathbb R}}

\newcommand{\1}{{\rm 1\hspace*{-0.4ex}%
\rule{0.1ex}{1.52ex}\hspace*{0.2ex}}}

\thanks{The author was supported in part by the NSF grant DMS-1361476.}

\begin{document}

\title[Improved critical eigenfunction estimates]
{Improved critical
eigenfunction   estimates on manifolds of nonpositive curvature}
%
%
%
%
%
\begin{abstract}  We prove new improved endpoint, $L^{p_c}$, $p_c=\tfrac{2(n+1)}{n-1}$, estimates (the ``kink point'') for eigenfunctions 
on manifolds of
nonpositive curvature.  We do this by using energy and
dispersive estimates for the wave equation as well as
new improved $L^p$, $2<p< p_c$, 
bounds
of  Blair and the author \cite{BSTop}, \cite{BSK15}
and the classical improved sup-norm estimates of
B\'erard~\cite{Berard}.  Our proof  uses Bourgain's
\cite{BKak} proof of weak-type estimates for the Stein-Tomas
Fourier restriction theorem \cite{Tomas}--\cite{Tomas2} as a template
to be able to obtain improved weak-type $L^{p_c}$ estimates under this geometric
assumption.  We can then use these estimates and the (local)  improved Lorentz space estimates
of Bak and Seeger~\cite{BakSeeg} (valid for all manifolds) to obtain our improved
estimates for the critical space under the assumption of nonpositive sectional curvatures.
\end{abstract}

\author{Christopher D. Sogge}
\address{Department of Mathematics,  Johns Hopkins University,
Baltimore, MD 21218}
\email{sogge@jhu.edu}

\maketitle

\newsection{Introduction}

Let $(M,g)$ be a compact $n$-dimensional Riemannian manifold and let $\Delta_g$ be the associated
Laplace-Beltrami operator.  We shall consider $L^2$-normalized eigenfunctions of frequency $\la$, i.e.,
$$-\Delta_g e_\la=\la^2 e_\la, \quad \int_M |e_\la|^2 \, dV_g=1,$$
with $dV_g$ denoting the volume element.

The author showed in \cite{Sef} that one has the following bounds for a given $2<p\le \infty$ and
$\la\ge1$:
\begin{equation}\label{1.1}
\|e_\la\|_{L^p(M)}\le C\la^{\mu(p)}, \quad
\mu(p)=\max\bigl(\tfrac{n-1}2(\tfrac12-\tfrac1p), \, n(\tfrac12-\tfrac1p)-\tfrac12\bigr).
\end{equation}
These estimates are saturated on the round sphere by zonal functions, $Z_\la$, for $p\ge \tfrac{2(n+1)}{n-1}=p_c$ and
for $2<p\le p_c$ by the highest weight spherical harmonics $Q_\la =\la^{\frac{n-1}4}(x_1+ix_2)^k$, if 
$\la=\la_k=\sqrt{(k+n-1)k}$.  See \cite{Sthesis}.  The zonal functions have the maximal concentration at points allowed
by the sharp Weyl formula, while the highest weight spherical harmonics have the maximal concentration near periodic
geodesics that is allowed by \eqref{1.1}.

Over the years there has been considerable work devoted to determining when \eqref{1.1} can be improved.  Although
not explicitly stated, this started in the work of B\'erard~\cite{Berard}, which implies that for manifolds of nonpositive
curvature the estimate for $p=\infty$ can be improved by a $(\log \la)^{-\frac12}$ factor (see \cite[Proposition 3.6.2]{Hang}).
By interpolation with the special case of $p=p_c$ in \eqref{1.1}, one obtains improvement for all exponents
$p_c<p\le \infty$, which was further recently improved by Hassell and Tacy~\cite{HassellTacy}.  The author and Zelditch~\cite{SZDuke}
showed that for generic manifolds one can obtain $o(\la^{\mu(p)})$ bounds for $\|e_\la\|_{L^p}$ if $p_c<p\le \infty$.  These
results were improved in \cite{STZ} and in \cite{SZRev} and \cite{SZRev2}.  In the latter two articles, a necessary and sufficient
condition in the real analytic setting was obtained for such bounds for exponents larger than the critical one, $p_c$.

The estimate for the complementary range of $2<p<p_c$ has also garnered much attention of late.  In works of
Bourgain~\cite{Bourgainef} and the author~\cite{SKN} for $n=2$, it was shown that improvements of \eqref{1.1} for this
range is equivalent to improvements of the geodesic restriction estimates of Burq, G\'erard and Tzvetkov~\cite{BGT}, as well as natural
Kakeya-Nikodym bounds introduced in \cite{SKN} measuring $L^2$-concentration of eigenfunctions on $\la^{-\frac12}$ tubes about
unit-length geodesics.  This is all very natural in view of the properties of the highest weight spherical harmonics
(see \cite{SKN} and \cite{SCon} for further discussion).  Using this equivalence and improved geodesic restriction estimates,
the author and Zelditch showed in \cite{SZStein} that $\|e_\la\|_{L^p}=o(\la^{\mu(p)})$ for $2<p<p_c$ if $n=2$ under the assumption
of nonpositive curvature, and similar improved bounds in higher dimensions and the equivalence of this problem and
improved Kakeya-Nikodym estimates were obtained by Blair and the author in \cite{BSJ}.  Very recently, in \cite{BSTop}
and \cite{BSK15}, we were able to obtain logarithmic improvements for this range of exponents in all dimensions under the
assumption of nonpositive curvature using 
microlocal analysis and the
classical Toponogov triangle comparison theorem in Riemannian geometry.
In addition to relationships with geodesic concentration and quantum ergodicity, improvements of \eqref{1.1} for
$2<p\le p_c$ are of interest because of their connection with nodal problems for eigenfunctions
(see, e.g., \cite{BSJ}, \cite{BSTop}, \cite{CM}, \cite{HezR}, \cite{HezS}, \cite{SZNod1} and \cite{SZNod2}).

Despite the success in obtaining improvements of \eqref{1.1} for the ranges $2<p<p_c$ and
$p_c<p\le \infty$, improvements for the critical space where $p=p_c=\tfrac{2(n+1)}{n-1}$ have proven to be
elusive.  The special case of \eqref{1.1} for this exponent reads as follows:
\begin{equation}\label{1.1'}\tag{1.1$'$}
\|e_\la\|_{L^{\frac{2(n+1)}{n-1}}(M)}\le C\la^{\frac{n-1}{2(n+1)}},
\end{equation}
and by interpolating with the trivial $L^2$ estimate and the sup-norm estimate
$\|e_\la\|_{L^\infty}=O(\la^{\frac{n-1}2})$, which is implicit in Avakumovi\'c~\cite{Av}
and Levitan~\cite{Lev}, one obtains all of the other bounds in \eqref{1.1}.

Improving \eqref{1.1'} has been challenging in part because it detects both point concentration and
concentration along periodic geodesics (as we mentioned for the sphere).  The techniques developed
for improving \eqref{1.1} for $p>p_c$ focused on the former and the more recent ones for
$2<p<p_c$ focused on the latter.  To date the only improvements of \eqref{1.1'} are recent
ones of Hezari and Rivi\`ere \cite{HezR} who used small-scale variants of the classical
quantum ergodic results of Colin de Verdi\`ere~\cite{CdVQE}, Snirelman~\cite{SnQE} and 
Zelditch~\cite{ZQE} (see also \cite{Zrate}) to show that for manifolds of strictly negative sectional
curvature there is a density one sequence of eigenfunctions for which \eqref{1.1'} can be 
logarithmically improved.  The $L^2$-improvements for small balls that were used had been
obtained independently by Han~\cite{Han} earlier, and, in
a companion article \cite{SHR} to \cite{HezR},
the author showed that, under the weaker assumption of ergodic geodesic flow, one can improve \eqref{1.1'}
for a density one sequence of eigenfunctions.

Our main result here is that, under the assumption of nonpositive curvature, one can obtain improved $L^{p_c}$
estimates for {\em all} eigenfunctions:


\begin{theorem}\label{thm1.1}
Assume that $(M,g)$ is of nonpositive curvature.  Then there is a constant $C=C(M,g)$ so that for $\la \gg 1$
\begin{equation}\label{main}  
\|e_\la\|_{L^{\frac{2(n+1)}{n-1}}(M)}\le C\la^{\frac{n-1}{2(n+1)}} \bigl(\log \log \la\bigr)^{-\frac2{(n+1)^2}}.
\end{equation}
Additionally, 
\begin{equation}\label{main2}
\bigl\|\chi_{[\la,\la+(\log\la)^{-1}]} f\bigr\|_{L^{\frac{2(n+1)}{n-1}}(M)} \le C
\la^{\frac{n-1}{2(n+1)}}\bigl(\log \log \la\bigr)^{-\frac2{(n+1)^2}} \|f\|_{L^2(M)}.
\end{equation}
\end{theorem}

Here
 if $0=\la_0<\la_1\le \la_2\le \cdots$ are the eigenvalues of $\sqrt{-\Delta_g}$ counted with respect to 
multiplicity and if $\{e_j\}$ is an associated orthonormal basis of eigenfunctions, 
if $I\subset [0,\infty)$
$$\chi_If=\sum_{\la_j\in I}E_jf,$$
where
$$E_j f(x)=\bigl(\int_M f \, \overline{e_j}\, dV_g\bigr) \times e_j(x),$$
denotes the projection onto the $j$th eigenspace.  Thus, \eqref{main2} implies \eqref{main}.

By interpolation and an application of a Bernstein inequality, 
this bound implies that for all exponents $p\in (2,\infty]$ one can improve
\eqref{1.1} by a power of $(\log \log\la)^{-1}$.  Although stronger log-improvements are in \cite{Berard},
\cite{BSTop}, \cite{BSK15} and \cite{HassellTacy} for $p\ne p_c$, \eqref{main} represents the first improvement involving all eigenfunctions for the critical exponent.
Also, besides the earlier improved geodesic eigenfunction restriction estimates for $n=2$ of Chen and the author~\cite{ChenS}, this result seems to be the first
improvement of estimates that are saturated by both the zonal functions and highest weight spherical harmonics on spheres.

The main step in proving these $L^{p_c}$-bounds will be to show that one has the following related weak-type estimates:

\begin{proposition}\label{prop1.2}
Assume, as above, that $(M,g)$ is a fixed manifold of nonpositive curvature.  Then there is a uniform constant
$C$ so that for $\la \gg 1$ we have
\begin{multline}\label{main2'}\tag{1.3$'$}
\bigl|\bigl\{x\in M: \bigl|\chi_{[\la,\la+(\log\la)^{-1}]}f(x)\bigr|>\alpha\bigr\}\bigr|
\le C\la \bigl(\log \log \la\bigr)^{-\frac2{n-1}}\alpha^{-\frac{2(n+1)}{n-1}},
\\
\alpha>0, \quad \text{if } \, \, \|f\|_{L^2(M)}=1.
\end{multline}
Here $|\Omega|$ denotes the $dV_g$ measure of a subset $\Omega$ of $M$.
\end{proposition}

Note that, by Chebyshev's inequality \eqref{main2} implies an inequality of the type \eqref{main2'}, but
with a less favorable exponent for the $\log \log \la$ factor.  The inequality says
that $\chi_{[\la,\la+(\log\la)^{-1}]}$ sends $L^2(M)$ into $L^{p_c,\infty}(M)$, i.e., weak-$L^{p_c}$,
with norm satisfying
\begin{equation}\label{main2''}\tag{1.3$''$}
\| \chi_{[\la,\la+(\log\la)^{-1}]}\|_{L^2(M)\to L^{\frac{2(n+1)}{n-1},\infty}(M)}=O\bigl(\la^{\frac{n-1}{2(n+1)}}/(\log\log \la)^{\frac1{n+1}}).
\end{equation}

After we obtain this weak-type $L^{p_c}$ estimate, we shall be able to obtain \eqref{main2} by, in effect,
interpolating it with another improved $L^{p_c}$ estimate of Bak and Seeger~\cite{BakSeeg}, which says
that the operators $\chi_{[\la,\la+1]}$ map $L^2(M)$ into the Lorentz space $L^{p_c,2}(M)$ 
(see \S 4 for definitions)
with norm
$O(\la^{\frac{n-1}{2(n+1)}})$.  This ``local'' estimate holds for all manifolds---no curvature assumption is needed.

Before turning to the proofs, let us point out that
the weak-type bound \eqref{main2'} cannot hold for $S^n$.  There there are two special values of $\alpha$
that cause problems.  The zonal functions are sensitive to $\alpha\approx \la^{\frac{n-1}2}$, and
$$|\{x\in S^n: \, |Z_\la(x)|>\alpha\}| \approx \la^{-n}\approx \la \alpha^{-\frac{2(n+1)}{n-1}},
\quad \text{if } \, \, \alpha =c\la^{\frac{n-1}2},$$
with $c>0$ fixed sufficiently small.
Similarly, the highest weight spherical harmonics, $Q_\la$, are sensitive to $\alpha \approx \la^{\frac{n-1}4}$ in that
$$|\{x\in S^n: \, |Q_\la(x)|>\alpha\}| \approx \la^{-\frac{n-1}2} \approx \la \alpha^{-\frac{2(n+1)}{n-1}}, 
\quad \text{if } \, \, \alpha =c\la^{\frac{n-1}4},$$
and $c>0$ fixed sufficiently small.   Note that by \eqref{1.1'} and Chebyshev's inequality, we always have, 
on any $(M,g)$, 
\begin{equation}\label{1.2}
|\{x\in M: \, |e_\la(x)|>\alpha\}|\lesssim \la \alpha^{-\frac{2(n+1)}{n-1}},
\end{equation}
and so the zonal functions and the highest weight spherical harmonics saturate this weak-type estimate.  We shall give
a simple proof of \eqref{1.2} in the next section that will serve as a model for the proof of the improved weak-type bounds in 
Proposition~\ref{prop1.2}.  It is based on a modification of Bourgain's~\cite{BKak} proof of a weak-type version of the 
critical Fourier restriction estimate of Stein and Tomas~\cite{Tomas}--\cite{Tomas2}.


Let us 
give an overview of 
why  are able to obtain \eqref{main2'} and \eqref{main2}.  As we mentioned before,
the potentially dangerous values of $\alpha$ for the former are $\alpha\approx \la^{\frac{n-1}2}$ and $\alpha\approx \la^{\frac{n-1}4}$.
The aforementioned sup-norm estimates of B\'erard~\cite{Berard} provide log-improvements over
\eqref{1.2} for $\alpha \ge \la^{\frac{n-1}2}/(\log \la)^{\frac12}$, while the recent log-improved $L^p$ estimates,
$2<p<p_c$, of Blair and the author \cite{BSTop}, \cite{BSK15} yield log-improvements for $\alpha$ near the other dangerous
value $\la^{\frac{n-1}4}$.   Specifically, we are able to obtain improvements
when $\alpha \le \la^{\frac{n-1}4}(\log \la)^{\delta_n}$ for some
$\delta_n>0$.
  We can cut and paste these improvements into 
  the aforementioned
  argument of 
Bourgain~\cite{BKak} to obtain \eqref{main2'}.
We then can upgrade the weak-type estimates that we obtain (at the expense of less favorable powers of 
$(\log \log \la)^{-1}$) to a standard $L^{p_c}$ estimate using the result of Bak and Seeger~\cite{BakSeeg}.
Thus, we combine the earlier ``global" results of \cite{Berard}, \cite{BSTop}, and \cite{BSK15}
with ``local" harmonic analysis techniques 
 to obtain
our main estimate \eqref{main2}.

The paper is organized as follows.  In the next section we shall give the variation of the argument from \cite{BKak} that
yields \eqref{1.2}.  In \S3 we shall show how we can use it along with the results of \cite{Berard}, \cite{BSTop} and
\cite{BSK15} to obtain Proposition~\ref{prop1.2}.  
Then in \S4 we shall give the simple proof showing that we can use it and the aforementioned result of Bak and Seeger~\cite{BakSeeg}
to obtain the Theorem.  Finally,
in \S5, we shall state some natural problems related to our approach.
Also, in what follows whenever we write $A\lesssim B$, we mean that $A$ is dominated by an unimportant constant multiplied by $B$.

%

\newsection{The model local argument}

%

In this section we shall present an argument that yields the weak-type estimate \eqref{1.2} and serves as a model
for the argument that we shall use to prove Theorem~\ref{thm1.1}.

Let us fix a real-valued function $\rho\in {\mathcal S}(\R)$ satisfying
\begin{equation}\label{2.1}
\rho(0)=1, \, \, \,  \, \,  |\rho(\tau)|\le 1, \, \, \,  \text{and} \, \, \text{supp } \Hat \rho \subset (-1/2,1/2).
\end{equation}
If we set 
$$P=\sqrt{-\Delta_g},$$
consider the operators
\begin{equation}\label{2.2}
\rho(\la-P)f(x)=\sum_{j=0}^\infty \rho(\la-\la_j)E_jf(x),\end{equation}
where, as before, $0=\la_0<\la_1\le \la_2\le \cdots$ are the eigenvalues counted with respect to multiplicity and 
$E_j$ denotes projection onto the $j$th eigenspace


The ``local'' analog of Proposition~\ref{prop1.2} then is the following result whose
proof we shall modify in the next section to obtain the ``global'' weak-type estimates \eqref{main2'}.

\begin{proposition}\label{prop2.1}  For $\la\ge1$ there is a constant $C$, depending only on $(M,g)$, so that
\begin{equation}\label{2.3}
\bigl|\{x\in M: \, |\rho(\la-P)f(x)|>\alpha \}\bigr| \le C\la \alpha^{-\frac{2(n+1)}{n-1}}\|f\|_{L^2(M)}^{\frac{2(n+1)}{n-1}}, \quad
\alpha>0.
\end{equation}
Consequently, \eqref{1.2} is valid, and, moreover,
if $\chi_\la$ denotes the unit-band spectral projection operators
$$\chi_\la f=\sum_{\la_j\in [\la,\la+1]}E_jf,$$
we have
\begin{equation}\label{2.3'}\tag{2.3$^\prime$}
\bigl|\{x\in M: \, |\chi_\la f(x)|>\alpha \}\bigr| \le C\la \alpha^{-\frac{2(n+1)}{n-1}}\|f\|_{L^2(M)}^{\frac{2(n+1)}{n-1}}, \quad
\alpha>0. 
\end{equation}
\end{proposition}

Since $\rho(0)=1$ we have that $|\rho(\tau)|\ge 1/2$ for $|\tau|\le \delta$ for some $\delta>0$.  Thus, if one applies
\eqref{2.3} with $f$ replaced by $\sum_{\la_j\in [\la,\la+\delta]}E_jf$, one deduces that
$$\bigl|\bigl\{ \, |\sum_{\la_j\in [\la,\la+\delta]}E_jf(x)|>\alpha\bigr\}\bigr| \le C\la \alpha^{-\frac{2(n+1)}{n-1}}
\|f\|_{L^2(M)}^{\frac{2(n+1)}{n-1}}, \quad \alpha>0,$$
which implies \eqref{2.3'}.
So to prove  Proposition~\ref{prop2.1}, we just need to prove \eqref{2.3}.  


To prove \eqref{2.3}, we require the following lemma which will be useful in the sequel.  We shall assume, as we may, here and in what follows
that the injectivity radius of $M$, $\text{Inj }M$, satisfies
$$\text{Inj }M \ge 10.$$
Also, $B(x,r)$, $r<\text{Inj }M$, denotes the geodesic ball of radius $r$ about a point $x\in M$ with
respect to the Riemannian distance function $d_g(\, \cdot \, ,\, \cdot\, )$.  The result we need then is the following.

\begin{lemma}\label{lemma2.2}  Let $a\in C^\infty_0((-1,1))$.  Then there is a constant $C$, depending only on $(M,g)$ and the size
of finitely many derivatives of $a$, so that for $\la^{-1}\le r\le {\rm Inj} \, M$ we have
\begin{equation}\label{2.5}
\Bigl\|\int a(t)e^{it\la} \bigl(e^{-itP}f\bigr) \, dt \Bigr\|_{L^2(B(x,r))}\le Cr^{\frac12}\|f\|_{L^2(M)},
\end{equation}
and, also, if $\bigl(e^{-itP}\bigr)(x,y)$ denotes the kernel of the half-wave operators $e^{-itP}$, we have
\begin{multline}\label{2.6}
\bigl|\bigl(\Hat a(P-\la)\bigr)(x,y)\bigr| = \Bigl|  \int a(t)e^{it\la} \bigl(e^{-itP}\bigr)(x,y) \, dt\Bigr|
\\
\le C\la^{\frac{n-1}2}\bigl(d_g(x,y)+\la^{-1}\bigr)^{-\frac{n-1}2}.
\end{multline}
\end{lemma}

We shall omit the proof of \eqref{2.6} since it is well known and follows easily from using stationary phase and
parametrices for the half-wave equation.  One can easily obtain \eqref{2.6} by adapting the proof of 
Lemma~5.1.3 in \cite{SFIO}.

Even though \eqref{2.5} is in a recent article of the author~\cite{SHR}, for the sake of completeness, we shall present
a different simple proof here, which only uses energy estimates and quantitative propagation of 
singularities estimates for the half-wave operators.

We start by introducing a Littlewood-Paley bump function $\beta\in C^\infty_0(\R)$ satisfying
\begin{equation}\label{2.7}
\beta(\tau)=1, \, \, \tau\in [1/2,2], \quad \text{and } \, \, \text{supp }\beta\subset (1/4,4).
\end{equation}
Then standard arguments using the aforementioned parametrix show that for any $N$, we have that
$$\Bigl\|\int a(t)e^{it\la}\bigl(I-\beta(P/\la)\bigr)\circ e^{-itP} \, dt\Bigr\|_{L^2(M)\to L^2(M)} = O(\la^{-N}),$$
where for each $N\in {\mathbb N}$ the constants depend only on finitely many derivatives of $a$.
Thus, to prove \eqref{2.5}, it suffices to prove the variant where $e^{-itP}$ is replaced by
$\beta(P/\la)\circ e^{-itP}$.  By a routine $TT^*$ argument, this in turn is equivalent to showing that
\begin{multline}\label{2.5'•}\tag{2.5$^\prime$}
\Bigl\|\int b(t) e^{it\la}\bigl(\beta(P/\la)\circ e^{-itP}\bigr)h \, dt \Bigr\|_{L^2(B(x,r))} \le Cr\|h\|_{L^2(B(x,r))},
\\
\text{if } \, \, \text{supp }h\subset B(x,r) \, \, \text{and } \, \, \la^{-1}\le r\le \text{Inj }M,
\end{multline}
with
$$b=a(\cd)*\overline{a(-\cd)}.$$

By Minkowski's inequality, the left side of \eqref{2.5'•} is dominated by
\begin{multline*}
\int_{|t|\le 10r}|b(t)| \, \bigl\| \bigl(\beta(P/\la)\circ e^{-itP}\bigr)h \,  \bigr\|_{L^2(B(x,r))}\, dt
\\
+\int_{|t|\ge 10r}|b(t)| \, \bigl\| \bigl(\beta(P/\la)\circ e^{-itP}\bigr)h \,  \bigr\|_{L^2(B(x,r))}\, dt
=I + II.
\end{multline*}
By energy estimates, we trivially have
$$I\lesssim r\|h\|_{L^2},$$
as desired, and we do not need to use our support assumptions in \eqref{2.5'•} here.

To handle $II$, though, we do need to make use of them.  We also need the routine dyadic
estimates
\begin{equation}\label{2.8}
\Bigl|\bigl(\beta(P/\la)\circ e^{-itP}\bigr)(w,z)\Bigr| =O\bigl(\la^n(1+\la|t|\bigr)^{-N}\bigr) \, \, \forall \, N, 
\quad \text{if } \, d_g(w,z)\le |t|/2,
\end{equation}
which also follows easily from an integration by parts argument using the parametrix for $e^{-itP}$.  From
\eqref{2.8} we immediately get
\begin{equation*}
\Bigl|\bigl(\beta(P/\la)\circ e^{-itP}\bigr)(w,z)\Bigr| =
O\bigl(\la^n(1+\la|t|)^{-N}\bigr) \, \forall N, \, \, 
\text{if }  \, w,z\in B(x,r) \, \, \text{and } \, \, |t|\ge 10r.
\end{equation*}
As a result, by Schwarz's inequality, we have that if, as in \eqref{2.5'•}, $\text{supp }h\subset B(x,r)$,
$$II\lesssim (r\la)^n\Bigl(\int_{|t|\ge 10r}\bigl(\la|t|\bigr)^{-n}\, dt\Bigr)\times \|h\|_{L^2} \approx r\|h\|_{L^2},
$$
as desired, completing the proof of \eqref{2.5'•}.

\begin{proof}[Proof of Proposition~\ref{prop2.1}]  To prove \eqref{2.3} it suffices to show that if $\Omega$
is a relatively compact subset of a coordinate patch $\Omega_0$ for $M$ then we have
\begin{equation}\label{2.9}
\bigl|\bigl\{x\in \Omega: \, |\rho(\la-P)f(x)|>\alpha \bigr\}\bigr| \le C\la \alpha^{-\frac{2(n+1)}{n-1}}, \quad \alpha>0,
\end{equation}
assuming that
\begin{equation}\label{2.10}
\|f\|_{L^2(M)}=1.
\end{equation}
We shall work in these local coordinates to make the decomposition we require.

Let 
$$A=\{x\in \Omega: \, |\rho(\la-P)f(x)|>\alpha\}$$
denote the set in \eqref{2.9}.  Our decomposition will be based on the scale
\begin{equation}\label{2.11}
r=\la \alpha^{-\frac4{n-1}},
\end{equation}
which is motivated by an argument in Bourgain~\cite{BKak}.  
Note that, since the sup-norm estimates of Avakumovi{\'c} \cite{Av} and Levitan \cite{Lev} give
$$\|\rho(\la-P)f\|_{L^\infty}=O(\la^{\frac{n-1}2}),$$
the estimate \eqref{2.9} is trivial when $r$ is smaller than a multiple of $\la^{-1}$, which allows us to use \eqref{2.5}.

Write
$$A=\bigcup A_j,$$
where $A_j=A\cap Q_j$ and $Q_j$ denote a nonoverlapping lattice of cubes of sidelength $r$ in our 
coordinates.  At the expense of replacing $A$ by a set of proportional measure, we may assume
that
\begin{equation}\label{2.12}
\text{dist } (A_j,A_k)>C_0 r, \quad j\ne k,
\end{equation}
for a constant $C_0$ to be specified later.  Also, let
\begin{equation}\label{2.13}
\psi_\la(x)=
\begin{cases} \rho(\la-P)f(x)/|\rho(\la-P)f(x)|, \, \, \text{if } \, \rho(\la-P)f(x)\ne 0
\\
1, \, \, \text{otherwise},
\end{cases}
\end{equation}
so that $\psi_\la$, of modulus one, is the signum function of $\rho(\la-P)f$.

We then have, by Chebyshev's inequality, \eqref{2.10} and the Cauchy-Schwarz inequality,
$$\alpha|A|
\le \Bigl|\int \rho(\la-P)f \, \,  \overline{\psi_\la \1_A} \, dV_g\Bigr|
\le \Bigl(\int \bigl|\sum_j \rho(\la-P)a_j\bigr|^2 \, dV_g\Bigr)^{\frac12},
$$
where $\1_A$ denotes the indicator function of $A$ and $a_j$ denotes $\psi_\la$ times the 
indicator function of $A_j$.  As a result, if $S_\la=\bigl(\rho(\la-P)^*\circ \rho(\la-P)\bigr)=\rho^2(\la-P)$,
\begin{align*}
\alpha^2|A|^2 &\le \sum_j \int |\rho(\la-P)a_j|^2 \, dV_g + \sum_{j\ne k}\int \rho(\la-P)a_j \, \overline{\rho(\la-P)a_k}\, dV_g
\\
&=\sum_j \int|\rho(\la-P)a_j|^2 \, dV_g+ \sum_{j\ne k} \int S_\la a_j \, \overline{a_k} \, dV_g
\\
&=I + II.
\end{align*}

Since $a_j$ is supported in a ball of radius $\approx r$, by \eqref{2.1} and the dual version of \eqref{2.5} with $a=\Hat \rho$, we have
$$\int |\rho(\la-P)a_j|^2 \, dV_g \le Cr\int |a_j|^2 \, dV_g =Cr|A_j|.$$
Whence, by \eqref{2.11}
$$I\lesssim r|A|=\la\alpha^{-\frac4{n-1}}|A|.$$

To estimate $II$, we note that by \eqref{2.6} with $a=\Hat \rho(\cd) *\overline{\Hat \rho(-\cd)}$, we have that the 
kernel $K_\la(x,y)$ of $S_\la$ satisfies
\begin{equation}\label{2.14}
|K_\la(x,y)|\le C\la^{\frac{n-1}2}\bigl(d_g(x,y)+\la^{-1}\bigr)^{-\frac{n-1}2}.
\end{equation}
Therefore, by \eqref{2.13},
\begin{align*}
II &\lesssim \sum_{j\ne k} \iint |K_\la(x,y)| \, |a_j(x)| \, |a_k(y)| \, dV_g(x)dV_g(y)
\\
&\lesssim \la^{\frac{n-1}2} \bigl(C_0r\bigr)^{-\frac{n-1}2} \sum_{j\ne k}\|a_j\|_{L^1}\|a_k\|_{L^1}
\\
&\le C_0^{-\frac{n-1}2} \alpha^2|A|^2.
\end{align*}
Thus, 
$$\alpha^2|A|^2 \lesssim \la\alpha^{-\frac4{n-1}} |A|+C_0^{-\frac{n-1}2}\alpha^2|A|^2,$$
and so, if $C_0$ in \eqref{2.12} is large enough, the last term can be absorbed in the left side.
We conclude that
$$|A|\lesssim \la \alpha^{-2-\frac4{n-1}}=\la \alpha^{-\frac{2(n+1)}{n-1}},$$
which is \eqref{2.9}.
\end{proof}

\newsection{Proof of improved weak-type estimates}

We shall now prove Proposition~\ref{prop1.2}.  Repeating the arguments from the previous section shows
that if $\rho\in {\mathcal S}(\R)$ is as in \eqref{2.1} then it suffices to show that we have the following

\begin{proposition}\label{prop3.1}  Let  $(M,g)$ be an $n$-dimensional compact Riemannian manifold of 
nonpositive curvature.  Then for $\la\gg 1$
\begin{equation}\label{3.1}
\bigl\|\rho(\log\la(\la-P))\bigr\|_{L^2(M)\to L^{\frac{2(n+1)}{n-1},\infty}(M)} =O\bigl(\la^{\frac{n-1}{2(n+1)}}/(\log\log\la)^{\frac1{n+1}}\bigr).
\end{equation}
\end{proposition}

The earlier arguments show that
\eqref{3.1} yields \eqref{main2''} and hence Proposition~\ref{prop1.2}
assuming, as in there and as we shall throughout this section, that the sectional curvatures of $(M,g)$ are nonpositive.

To prove \eqref{3.1}, as in \eqref{2.9}, it suffices to show now that if $\Omega$ is a relatively compact subset of a coordinate
patch $\Omega_0$, then
\begin{equation}\label{3.2}
\bigl|\bigl\{x\in \Omega: \, |\rho(\log\la(\la-P))f(x)|>\alpha\bigr\}\bigr| \le C\alpha^{-\frac{2(n+1)}{n-1}} \la/(\log\log\la)^{\frac2{n-1}},
\end{equation}
assuming that
\begin{equation}\label{3.3}
\|f\|_{L^2(M)}=1.
\end{equation}

To prove this, in addition to \eqref{2.5}, we shall require the following two results.

\begin{lemma}\label{lemma3.2}  Let $(M,g)$ be as above.  Then there is a $\delta_n>0$ so that for $\la\gg 1$ and
$\mu(p)$ as in \eqref{1.1}
\begin{equation}\label{3.4}
\bigl\|\rho(\log\la(\la-P))\bigr\|_{L^2(M)\to L^{\frac{2n}{n-1}}(M)}=O\bigl(\la^{\mu\bigl(\tfrac{2n}{n-1}\bigr)}/(\log\la)^{\delta_n}).
\end{equation}
\end{lemma}

\begin{lemma}\label{lemma3.3}  If $(M,g)$ is as above then there is a constant $C=C(M,g)$ so that for $T\ge 1$ and large
$\la$ we have the following bounds for the kernel of $\eta(T(\la-P))$, $\eta=\rho^2$,
\begin{equation}\label{3.5}
\bigl|\eta \bigl(T(\la-P)\bigr)(w,z)\bigr| \le CT^{-1}(\la/d_g(w,z))^{\frac{n-1}2} + C\la^{\frac{n-1}2}\exp(CT).
\end{equation}
\end{lemma}

The first estimate, \eqref{3.4}, is a simple consequence of the bounds
\begin{equation}\label{3.4'}\tag{3.4$'$}
\bigl\|\chi_{[\la,\la+(\log\la)^{-1}]}\bigr\|_{L^2(M)\to L^p(M)}\le \la^{\mu(p)}/(\log\la)^{\delta(p,n)}, \quad
2<p<\tfrac{2(n+1)}{n-1},
\end{equation}
with $\delta(p,n)>0$ from \cite{BSTop} for the special case of $p=\tfrac{2n}{n-1}$.  Any other exponent between
2 and $\tfrac{2(n+1)}{n-1}$ in \eqref{3.4'} would work as well for us.  We just chose $p=\tfrac{2n}{n-1}$ to
simplify the calculations.

The other bound, \eqref{3.5}, is well known and follows from the arguments in B\'erard~\cite{Berard}.  Indeed, it is 
a simple consequence of inequality (3.6.8) in \cite{Hang}.

Let us see how we can use these results to obtain \eqref{3.2}.

We first note that by Lemma~\ref{lemma3.2} and the Chebyshev inequality we have that
since $\tfrac{2n}{n-1}\cdot \mu(\tfrac{2n}{n-1})=\tfrac12$,
\begin{align}\label{3.6}
\bigl|\bigl\{x\in \Omega: \, |\rho(\log\la(\la-P))f(x)|>\alpha\bigr\}\bigr|
&\le \alpha^{-\frac{2n}{n-1}}\int_M |\rho(\log\la(\la-P))f|^{\frac{2n}{n-1}} \, dV_g
\\
&\lesssim \alpha^{-\frac{2n}{n-1}}\la^{\frac12}(\log\la)^{-\frac{2n}{n-1}\delta_n}.\notag
\end{align}

To use this, we note that for large $\la$ we have
\begin{equation}\label{3.7}
\alpha^{-\frac{2n}{n-1}}\la^{\frac12}(\log\la)^{-\frac{2n}{n-1}\delta_n}\ll \alpha^{-\frac{2(n+1)}{n-1}} \la \bigl(\log\log\la\bigr)^{-\frac2{n-1}},
\quad
\text{if } \, \, \, \, \alpha\le \la^{\frac{n-1}4}(\log\la)^{\delta_n}.
\end{equation} 
Thus, by \eqref{3.6}, we would obtain \eqref{3.2} if we could show that for $\la\gg 1$
\begin{multline}\label{3.8}
\bigl|\bigl\{x\in \Omega: \, |\rho(\log\la(\la-P))f(x)|>\alpha\bigr\}\bigr|
\le C \alpha^{-\frac{2(n+1)}{n-1}} \la (\log\log\la)^{-\frac2{n-1}}, 
\\
\text{if } \, \, \alpha\ge \la^{\frac{n-1}4}(\log\la)^{\delta_n}.
\end{multline}

As we mentioned in the introduction, this step is key for us since it has allowed us to use our curvature assumptions
and move  well past the dangerous heights where $\alpha$ is comparable to $\la^{\frac{n-1}4}$.

At this stage, due to the nature of the pointwise estimates in Lemma~\ref{lemma3.3}, we need to change the
frequency scale we are working with.  Instead of effectively working with $(\log\la)^{-1}$ windows for
frequencies as above, we shall work with wider windows of size $T^{-1}$ where $T=c_0\log\log\la$, with $c_0$
chosen later to deal with the second term in the right side of \eqref{3.5}.

We claim that we would have \eqref{3.8}, and therefore be done, if we could show that
\begin{multline}\label{3.9}
\bigl|\bigl\{x\in \Omega: \, |\rho\bigl(c_0\log\log\la(\la-P)\bigr)h(x)|>\alpha\bigr\}\bigr| \lesssim \alpha^{-\frac{2(n+1)}{n-1}}\la (\log\log\la)^{-\frac1{n+1}},
\\
\text{if } \, \, \, \alpha\ge \la^{\frac{n-1}4}(\log\la)^{\delta_n}, \, \, \, \text{and } \, \, \|h\|_{L^2(M)}\le 1.
\end{multline}

To verify this claim, we note that since $\rho(0)=1$ and $\rho\in {\mathcal S}$, for $\tau\in \R$ and for $\la\gg 1$ have
$$\bigl| \bigl[\rho(c_0\log\log\la(\la-\tau))-1\bigr]\rho(\log\la(\la-\tau))\bigr|
\lesssim \frac{\log\log\la}{\log\la} (1+|\la-\tau|)^{-N},
$$
for any $N=1,2,\dots$.  Thus, by using the fact that by \cite{Sef} the unit band spectral projection operators $\chi_\la$ satisfy
$$\|\chi_\la\|_{L^2(M)\to L^{\frac{2(n+1)}{n-1}}(M)}=O(\la^{\frac{n-1}{2(n+1)}}),$$
we deduce that 
$$\bigl\|\bigl[\rho(c_0\log\log\la(\la-P))-I\bigr]\circ \rho(\log\la(\la-P))f\|_{L^{\frac{2(n+1)}{n-1}}(M)}
\lesssim \frac{\log\log\la}{\log\la}\la^{\frac{n-1}{2(n+1)}},
$$
and so, by Chebyshev, for all $\alpha>0$ we have
\begin{multline*}
\bigl|\bigl\{x\in M: \, |[\rho(c_0\log\log\la(\la-P) )-I] \circ \rho(\log\la(\la-P))f(x)|>\alpha\bigr\}\bigr|
\\
\lesssim \bigl(\tfrac{\log\log\la}{\log\la}\bigr)^{\frac{2(n+1)}{n-1}}\la \alpha^{-\frac{2(n+1)}{n-1}},
\end{multline*}
which is much better than the bounds posited in \eqref{3.8}.  If we take $h=\rho(\log\la(\la-P))f$ in \eqref{3.9}, we deduce the claim from
this since, by \eqref{2.1}, $\|\rho(\log\la(\la-P))\|_{L^2(M)\to L^2(M)}\le1$.

Following the arguments from the preceding section, to prove \eqref{3.9}, let
$$A=\{x\in \Omega: \, |\rho(c_0\log\log\la(\la-P))h(x)|>\alpha\},$$
and let $\psi_\la$ be defined as in \eqref{2.13} but with $\rho(\la-P)$ replaced by
$\rho(c_0\log\log(\la-P))$.  Note that for large $\la$ 
$$A=\emptyset \quad \text{if } \quad
\quad \la^{\frac{n-1}2}(\log \log \la)^{-\frac12}\lesssim \alpha,
$$
since estimates of B\'erard~\cite{Berard} (see also \cite{Hang}) give
$$\|\rho(c_0\log\log\la(\la-P))\|_{L^2(M)\to L^\infty(M)}
\lesssim \la^{\frac{n-1}2}/\bigl(\log\log\la\bigr)^{\frac12}.
$$
This will allow us to apply \eqref{2.5}.

Next, as in the proof of Proposition~\ref{prop2.1}, we write $A=\cup A_j$ where
$A_j=Q_j\cap A$, with the $Q_j$ coming from a lattice of nonoverlapping cubes in
our coordinate system, except now, instead of \eqref{2.1}, we take
\begin{equation}\label{3.10}
r=\la \alpha^{-\frac4{n-1}}(\log\log\la)^{-\frac2{n-1}}.
\end{equation}
As before, at the expense of replacing $A$ by a set of proportional measure, we may assume that
\begin{equation}\label{3.11}
\text{dist }(A_j,A_k)>C_0r, \quad j\ne k,
\end{equation}
where $C_0$ will be specified momentarily. 

Let us now collect the two estimates that we need for the proof of \eqref{3.9}.  First, if
$S_\la =\eta(c_0\log\log\la(\la-P))$, $\eta=\rho^2$, then by \eqref{3.5} if $c_0>0$ is fixed small
enough, we have that its kernel, $K_\la$, satisfies
\begin{equation}\label{3.12}
|K_\la(w,z)|\le C\Bigl[
(\log\log\la)^{-1}\Bigl(\frac\la{d_g(w,z)}\Bigr)^{\frac{n-1}2}+\la^{\frac{n-1}2}(\log\la)^{\frac{\delta_n}{10}}\Bigr],
\end{equation}
with $C$ independent of $\la\gg 1$.  

The other estimate that we require is that there is a uniform constant so that, for $T\ge 1$, we have
\begin{equation}\label{3.13}
\bigl\|\rho(T(\la-P))f\|_{L^2(B(x,r))}\le Cr^{\frac12}\|f\|_{L^2(M)}, \quad \text{if } \, \, 
\la^{-1}\le r\le \text{Inj }M,
\end{equation}
with $C$ independent of $\la\gg 1$.
Since
$$\rho(T(\la-P))=\frac1{2\pi T}\int \Hat \rho(t/T) e^{it\la}e^{-itP}\, dt,$$
and, by \eqref{2.1}, $\Hat \rho(t/T)=0$ if $|t|\ge T$, this follows easily from \eqref{2.5} and the
fact that the half-wave operators $e^{-itP}$ are unitary.

We now use the proof of Proposition~\ref{prop2.1} to obtain \eqref{3.9}.  We argue as before to see that
if $T_\la =\rho(c_0\log\log\la(\la-P))$ and $a_j=\psi_\la\times \1_{A_j}$, then since $\|h\|_{L^2(M)}\le1$,
we have
$$\alpha^2|A|^2\le \sum_j \int |T_\la a_j|^2 \, dV_g
+\sum_{j\ne k}\int S_\la a_j \, \overline{a_k}\, dV_g = I + II.
$$
By the dual version of \eqref{3.13} and \eqref{3.10}
$$I\lesssim r\sum_j\int |a_j|^2 \, dV_g=r|A|=\la(\log\log\la)^{-\frac2{n-1}}\alpha^{-\frac4{n-1}}|A|.$$
By \eqref{3.12}
\begin{align*}
II&\lesssim \Bigl[(\log\log\la)^{-1} \la^{\frac{n-1}2}\bigl(C_0r\bigr)^{-\frac{n-1}2}+\la^{\frac{n-1}2}(\log\la)^{\frac{\delta_n}{10}}\Bigr] \,
\sum_{j\ne k}\|a_j\|_{L^1}\|a_k\|_{L^1}
\\
&\le C_0^{-\frac{n-1}2}\alpha^2|A|^2+\la^{\frac{n-1}2}(\log\la)^{\frac{\delta_n}{10}}|A|^2.
\end{align*}
Since we are assuming that $\alpha\ge \la^{\frac{n-1}4}(\log\la)^{\delta_n}$, the last term is $\ll \alpha^2|A|^2$ if $\la$ is large.
This means that we can fix $C_0$ in \eqref{3.11} so that for large $\la$ we have
$$II\le \frac12\alpha^2|A|^2.$$
Hence
$$\alpha^2|A|^2\le C\la(\log\log\la)^{-\frac2{n-1}}\alpha^{-\frac4{n-1}}|A|+\frac12\alpha^2|A|^2,$$
which of course yields the desired estimate
$$|A|\lesssim \la(\log\log\la)^{-\frac2{n-1}}\alpha^{-2-\frac4{n-1}}=
\la (\log\log\la)^{-\frac2{n-1}}\alpha^{-\frac{2(n+1)}{n-1}},$$
assuming, as we are, that $\alpha \ge \la^{\frac{n-1}4}(\log\la)^{\delta_n}.$

This concludes the proof of \eqref{3.9}, Proposition~\ref{prop3.1} and Proposition~\ref{prop1.2}.

\newsection{Proof of Theorem~\ref{thm1.1}}

Even though \eqref{main2}, and hence Theorem~\ref{thm1.1}, follows directly from interpolating between
the weak-type estimate \eqref{main2'}  and the estimate,
\begin{equation}\label{i.1}
\|\chi_{[\la,\la+1]}\|_{L^2(M)\to L^{p_c,2}(M)}=O(\la^{\frac1{p_c}}), \quad p_c=\tfrac{2(n+1)}{n-1},
\end{equation}
of Bak and Seeger~\cite{BakSeeg}, for the sake of completeness, we shall give the simple argument here.

Let us start by recalling some basic facts about Lorentz spaces. See \S3 in Chapter 5 of Stein and Weiss~\cite{SteinWeiss}
for more details.

First, given a function $u$ on $M$, we let
$$\omega(\alpha)=\bigl|\bigl\{x\in M: \, |u(x)|>\alpha\bigr\}\bigr|, \quad \alpha>0,$$
denote its distribution function, and
$$u^*(t)=\inf \{\alpha: \, \omega(\alpha)\le t\}, \quad t>0,$$
the nonincreasing rearrangement of $u$.

Then the Lorentz spaces $L^{p,q}(M)$ for $1\le p<\infty$ and $1\le q<\infty$ are defined as
all $u$ so that
\begin{equation}\label{i.2}
\|u\|_{L^{p,q}(M)}=\left( \, \frac{q}{p} \int_0^\infty \bigl[ \,  t^{\frac1p}u^*(t) \, \bigr]^q \, \frac{dt}t \, \right)^{\frac1q}<\infty.
\end{equation}
By equation (3.9) in Chapter 5 of \cite{SteinWeiss}, we then have
\begin{equation}\label{i.3}
\|u\|_{L^{p,p}(M)}=\|u\|_{L^p(M)},
\end{equation}
and by Lemma 3.8 there we also have
\begin{align*}
\sup_{t>0} t^{\frac1p}u^*(t)&=\sup_{\alpha>0} \alpha \bigl[ \,  \omega(\alpha) \, \bigr]^{\frac1p}
\\
&=\sup_{\alpha>0} \alpha \,  \bigl|\bigl\{ x\in M: \, |u(x)|>\alpha \bigr\}\bigr|^{\frac1p}. \end{align*}

If we take $u=\chi_{[\la,\la+(\log\la)^{-1}]}f$ and assume from now on that $\|f\|_{L^2(M)}=1$, we therefore
have, by our improved weak-type estimates \eqref{main2'},
\begin{equation}\label{i.4}
    \sup_{t>0} t^{\frac1{p_c}} u^*(t)\le C
\la^{\frac1{p_c}}\bigl(\log \log \la\bigr)^{-\frac1{n+1}}.
\end{equation}
Also, for this $u$ we have $\chi_{[\la,\la+1]}u=u$, and so, by \eqref{i.1},
\begin{equation}\label{i.5}
\|u\|_{L^{p_c,2}(M)}\le C\la^{\frac1{p_c}}\|u\|_{L^2(M)} \le C\la^{\frac1{p_c}}\|f\|_{L^2(M)}
= C\la^{\frac1{p_c}}.
\end{equation}

By \eqref{i.2}--\eqref{i.3} and \eqref{i.4}--\eqref{i.5}, we therefore get
\begin{align*}
    \|u\|_{L^{p_c}(M)}&=\left( \, \int_0^\infty \bigl[ \,  t^{\frac1{p_c}} u^*(t) \, \bigr]^{p_c} \, \frac{dt}t\, \right)^{\frac1{p_c}}
\\
&\le (p_c/2)^{\frac1{p_c}} \, \left( \, \sup_{t>0} t^{\frac1{p_c}} u^*(t)\right)^{\frac{p_c-2}{p_c}} \,
\left(\frac2{p_c} \int_0^\infty \bigl[ \,  t^{\frac1{p_c}}u^*(t) \, \bigr]^{2} \, \frac{dt}t\, \right)^{\frac1{p_c}}
\\
&\lesssim \bigl[\, \la^{\frac1{p_c}} \bigl(\log \log \la\bigr)^{-\frac1{n+1}} \, \bigr]^{\frac{p_c-2}{p_c}} \, \|u\|_{L^{p_c,2}(M)}^{\frac2{p_c}}
\\
&\lesssim \bigl[\, \la^{\frac1{p_c}} \bigl(\log \log \la\bigr)^{-\frac1{n+1}} \, \bigr]^{\frac{p_c-2}{p_c}} \, \left( \, \la^{\frac1{p_c}}\, \right)^{\frac2{p_c}}
\\
&=\la^{\frac1{p_c}} \, \bigl(\log\log \la\bigr)^{-\frac2{(n+1)^2}},
\end{align*}
as $(p_c-2)/(n+1)p_c=2/(n+1)^2$.  Since $u=\chi_{[\la,\la+(\log \la)^{-1}]} f$ and we are assuming that $\|f\|_{L^2(M)}=1$, 
we conclude that \eqref{main2} must be valid, which completes the proof of Theorem~\ref{thm1.1}. \qed

\newsection{Concluding remarks}

First of all, we were only able to get endpoint results with gains of powers of $\log\log\la$ instead of 
powers of $\log\la$ due to the estimate \eqref{3.5} for the smoothed out spectral projection kernels.
Ideally, one would want to be able to use a variant of \eqref{3.5} where the exponential factor
is not present for the second term in the right. 
 Lower bounds of Jakobson and Polterovich~\cite{JP0}--\cite{JP}
show that this error term cannot be $O(\la^{\frac{n-1}2})$, but their bounds do not rule out some
improvement over \eqref{3.5}, which would lead to more favorable  estimates.  

A better avenue for improvement, though, might be to try to improve the ball-localized estimates
\eqref{2.5},
where the operators $\Hat a (P-\la)$ are replaced by $\rho(T(\la-P)))$ for appropriate
$T=T(r)$.  A seemingly modest improvement where $r^{\frac12}$ is replaced by
$r^{\frac12}/(\log \la)^\e$, for some $\e>0$, if $\la^{-1} \le r\le (\log \la)^{-\delta}$, for some $\delta>0$
could be of use.
The author in \cite{SCon} obtained
such improvements with $\e=\tfrac12$ if $\la^{-1}\le r\ll \la^{-\frac12}$, but this does not seem very useful.
On the other hand, assuming that the curvature is strictly negative, Han~\cite{Han} and Hezari
and Rivi\`ere~\cite{HezR} obtained these types of bounds with $\e=n/2$ and $\delta$ depending on the
dimension for a density one sequence of eigenfunctions.  For toral eigenfunctions, Lester and Rudnick~\cite{LRud}
did even better for a density one sequence of eigenfunctions by showing, for instance,  that 
in when $n=2$ one can replace
$r^{\frac12}$ in \eqref{2.5} by $r^{\frac{n}2}$ all the way down to $r$ being equal to the essentially the
wavelength, i.e., $\la^{-1+o(1)}$ as $\la \to \infty$.  (See also \cite{HezR2} for earlier work.)

Finally, the arguments we have given could possibly prove new sharp bounds for eigenfunctions on manifolds
with boundary.  Sharp estimates in the two-dimensional case were obtained by Smith and the author~\cite{SSActa},
but sharp estimates in higher dimensions are only known for certain exponents.  It turns out that the critical
exponent for manifolds with boundary should be $\frac{6n+4}{3n-4}$, which is larger than the one for
the boundaryless case, $\frac{2(n+1)}{n-1}$.  

If one could obtain the analog of \eqref{2.6} in this setting with the right hand side replaced by
$$(\la/\text{dist }(x,y))^{\frac{n-1}2+\frac16},$$
then one would likely be able to obtain sharp weak-type estimates for $p=\frac{6n+4}{3n-4}$, which by interpolation
would yield sharp $L^p$ estimates for all other $p\in (2,\infty]$.  
One would also need analogs of \eqref{2.5}, but these are probably
much easier and likely follow from stretching arguments of Ivri{\u\i}~\cite{Iv} and Seeley~\cite{Seeley}.
In the model case involving the Friedlander model, recently Ivanovici, Lebeau and Planchon~\cite{ILP} obtained
dispersive estimates for wave equations which have similarities with the types of spectral projection kernel
estimates we just described.

\section*{Acknowledgements}
We are  grateful for  helpful suggestions and comments from our colleagues M. Blair, H. Hezari, A. Seeger and S. Zeldtich.

\bibliography{EF}{}

\def\cprime{$'$} \def\cprime{$'$}
\providecommand{\bysame}{\leavevmode\hbox to3em{\hrulefill}\thinspace}
\providecommand{\MR}{\relax\ifhmode\unskip\space\fi MR }
\providecommand{\MRhref}[2]{%
  \href{http://www.ams.org/mathscinet-getitem?mr=#1}{#2}
}
\providecommand{\href}[2]{#2}
\begin{thebibliography}{10}

\bibitem{Av}
V.~G. Avakumovi{\'c}, \emph{\"{U}ber die {E}igenfunktionen auf geschlossenen
  {R}iemannschen {M}annigfaltigkeiten}, Math. Z. \textbf{65} (1956), 327--344.

\bibitem{BakSeeg}
J.-G. Bak and A.~Seeger, \emph{Extensions of the {S}tein-{T}omas theorem},
  Math. Res. Lett. \textbf{18} (2011), no.~4, 767--781.

\bibitem{Berard}
P.~H. B{\'e}rard, \emph{On the wave equation on a compact {R}iemannian manifold
  without conjugate points}, Math. Z. \textbf{155} (1977), no.~3, 249--276.

\bibitem{BSTop}
M.~D. Blair and C.~D. Sogge, \emph{Concerning {T}oponogov's theorem and
  logarithmic improvement of estimates of eigenfunctions},  (2015),
  arXiv:1510.07726.

\bibitem{BSJ}
\bysame, \emph{On {K}akeya--{N}ikodym averages, {$L\sp p$}-norms and lower
  bounds for nodal sets of eigenfunctions in higher dimensions}, J. Eur. Math.
  Soc. (JEMS) \textbf{17} (2015), no.~10, 2513--2543.

\bibitem{BSK15}
\bysame, \emph{Refined and microlocal {K}akeya-{N}ikodym bounds of
  eigenfunctions in higher dimensions},  (2015), arXiv:1510.07724.

\bibitem{BKak}
J.~Bourgain, \emph{Besicovitch type maximal operators and applications to
  {F}ourier analysis}, Geom. Funct. Anal. \textbf{1} (1991), no.~2, 147--187.

\bibitem{Bourgainef}
\bysame, \emph{Geodesic restrictions and {$L\sp p$}-estimates for
  eigenfunctions of {R}iemannian surfaces}, Linear and complex analysis, Amer.
  Math. Soc. Transl. Ser. 2, vol. 226, Amer. Math. Soc., Providence, RI, 2009,
  pp.~27--35.

\bibitem{BGT}
N.~Burq, P.~G{\'e}rard, and N.~Tzvetkov, \emph{Restrictions of the
  {L}aplace-{B}eltrami eigenfunctions to submanifolds}, Duke Math. J.
  \textbf{138} (2007), no.~3, 445--486.

\bibitem{ChenS}
X.~Chen and C.~D. Sogge, \emph{A few endpoint geodesic restriction estimates
  for eigenfunctions}, Comm. Math. Phys. \textbf{329} (2014), no.~2, 435--459.

\bibitem{CM}
T.~H. Colding and W.~P. Minicozzi, II, \emph{Lower bounds for nodal sets of
  eigenfunctions}, Comm. Math. Phys. \textbf{306} (2011), no.~3, 777--784.

\bibitem{CdVQE}
Y.~Colin~de Verdi{\`e}re, \emph{Ergodicit\'e et fonctions propres du
  laplacien}, Comm. Math. Phys. \textbf{102} (1985), no.~3, 497--502.

\bibitem{Han}
X.~Han, \emph{Small scale quantum ergodicity in negatively curved manifolds},
  (2015), arXiv:1410.3911.

\bibitem{HassellTacy}
A.~Hassell and M.~Tacy, \emph{Improvement of eigenfunction estimates on
  manifolds of nonpositive curvature}, Forum Math. \textbf{27} (2015), no.~3,
  1435--1451.

\bibitem{HezR}
H.~Hezari and G.~Rivi\`ere, \emph{{$L^p$} norms, nodal sets, and quantum
  ergodicity}, Adv. Math. (2016), to appear.

\bibitem{HezR2}
\bysame, \emph{Quantitative equidistribution properties of toral
  eigenfunctions}, J. Spec. Theory (2016), to appear.

\bibitem{HezS}
H.~Hezari and C.~D. Sogge, \emph{A natural lower bound for the size of nodal
  sets}, Anal. PDE \textbf{5} (2012), no.~5, 1133--1137.

\bibitem{ILP}
O.~Ivanovici, G.~Lebeau, and F.~Planchon, \emph{Dispersion for the wave
  equation inside strictly convex domains {I}: the {F}riedlander model case},
  Ann. of Math. (2) \textbf{180} (2014), no.~1, 323--380.

\bibitem{Iv}
V.~Ja. Ivri{\u\i}, \emph{The second term of the spectral asymptotics for a
  {L}aplace-{B}eltrami operator on manifolds with boundary}, Funktsional. Anal.
  i Prilozhen. \textbf{14} (1980), no.~2, 25--34.

\bibitem{JP0}
D.~Jakobson and I.~Polterovich, \emph{Lower bounds for the spectral function
  and for the remainder in local {W}eyl's law on manifolds}, Electron. Res.
  Announc. Amer. Math. Soc. \textbf{11} (2005), 71--77.

\bibitem{JP}
\bysame, \emph{Estimates from below for the spectral function and for the
  remainder in local {W}eyl's law}, Geom. Funct. Anal. \textbf{17} (2007),
  no.~3, 806--838.

\bibitem{LRud}
S.~Lester and Z.~Rudnick, \emph{Small scale equidistribution of eigenfunctions
  on the torus},  (2015), arXiv:1508.01074.

\bibitem{Lev}
B.~M. Levitan, \emph{On the asymptotic behavior of the spectral function of a
  self-adjoint differential equation of the second order}, Izvestiya Akad. Nauk
  SSSR. Ser. Mat. \textbf{16} (1952), 325--352.

\bibitem{Seeley}
R.~Seeley, \emph{An estimate near the boundary for the spectral function of the
  {L}aplace operator}, Amer. J. Math. \textbf{102} (1980), no.~5, 869--902.

\bibitem{SSActa}
H.~F. Smith and C.~D. Sogge, \emph{On the {$L\sp p$} norm of spectral clusters
  for compact manifolds with boundary}, Acta Math. \textbf{198} (2007), no.~1,
  107--153.

\bibitem{SnQE}
A.~I. {\v{S}}nirel{\cprime}man, \emph{Ergodic properties of eigenfunctions},
  Uspehi Mat. Nauk \textbf{29} (1974), no.~6(180), 181--182.

\bibitem{Sthesis}
C.~D. Sogge, \emph{Oscillatory integrals and spherical harmonics}, Duke Math.
  J. \textbf{53} (1986), no.~1, 43--65.

\bibitem{Sef}
\bysame, \emph{Concerning the {$L\sp p$} norm of spectral clusters for
  second-order elliptic operators on compact manifolds}, J. Funct. Anal.
  \textbf{77} (1988), no.~1, 123--138.

\bibitem{SFIO}
\bysame, \emph{Fourier integrals in classical analysis}, Cambridge Tracts in
  Mathematics, vol. 105, Cambridge University Press, Cambridge, 1993.

\bibitem{SKN}
\bysame, \emph{Kakeya-{N}ikodym averages and {$L\sp p$}-norms of
  eigenfunctions}, Tohoku Math. J. (2) \textbf{63} (2011), no.~4, 519--538.

\bibitem{Hang}
\bysame, \emph{Hangzhou lectures on eigenfunctions of the {L}aplacian}, Annals
  of Mathematics Studies, vol. 188, Princeton University Press, Princeton, NJ,
  2014.

\bibitem{SCon}
\bysame, \emph{Problems related to the concentration of eigenfunctions},
  (2015), arXiv:1510.07723, Journe{e}s EDP, to appear.

\bibitem{SHR}
\bysame, \emph{Localized {$L^p$}-estimates of eigenfunctions: A note on an
  article of {Hezari and Rivi\`ere}}, Adv. Math. \textbf{289} (2016),
  384–--396.

\bibitem{STZ}
C.~D. Sogge, J.~A. Toth, and S.~Zelditch, \emph{About the blowup of quasimodes
  on {R}iemannian manifolds}, J. Geom. Anal. \textbf{21} (2011), no.~1,
  150--173.

\bibitem{SZDuke}
C.~D. Sogge and S.~Zelditch, \emph{Riemannian manifolds with maximal
  eigenfunction growth}, Duke Math. J. \textbf{114} (2002), no.~3, 387--437.

\bibitem{SZNod1}
\bysame, \emph{Lower bounds on the {H}ausdorff measure of nodal sets}, Math.
  Res. Lett. \textbf{18} (2011), no.~1, 25--37.

\bibitem{SZNod2}
\bysame, \emph{Lower bounds on the {H}ausdorff measure of nodal sets {II}},
  Math. Res. Lett. \textbf{19} (2012), no.~6, 1361--1364.

\bibitem{SZStein}
\bysame, \emph{On eigenfunction restriction estimates and {$L\sp 4$}-bounds for
  compact surfaces with nonpositive curvature}, Advances in analysis: the
  legacy of {E}lias {M}. {S}tein, Princeton Math. Ser., vol.~50, Princeton
  Univ. Press, Princeton, NJ, 2014, pp.~447--461.

\bibitem{SZRev}
\bysame, \emph{Focal points and sup-norms of eigenfunctions}, Rev. Mat.
  Iberomericana (2015), to appear.

\bibitem{SZRev2}
\bysame, \emph{Focal points and sup-norms of eigenfunctions {II}: the
  two-dimensional case}, Rev. Mat. Iberomericana (2015), to appear.

\bibitem{SteinWeiss}
E.~M. Stein and G.~Weiss, \emph{Introduction to {F}ourier analysis on
  {E}uclidean spaces}, Princeton University Press, Princeton, N.J., 1971,
  Princeton Mathematical Series, No. 32.

\bibitem{Tomas}
P.~A. Tomas, \emph{A restriction theorem for the {F}ourier transform}, Bull.
  Amer. Math. Soc. \textbf{81} (1975), 477--478.

\bibitem{Tomas2}
\bysame, \emph{Restriction theorems for the {F}ourier transform}, Harmonic
  analysis in {E}uclidean spaces ({P}roc. {S}ympos. {P}ure {M}ath., {W}illiams
  {C}oll., {W}illiamstown, {M}ass., 1978), {P}art 1, Proc. Sympos. Pure Math.,
  XXXV, Part, Amer. Math. Soc., Providence, R.I., 1979, pp.~111--114.

\bibitem{ZQE}
S.~Zelditch, \emph{Uniform distribution of eigenfunctions on compact hyperbolic
  surfaces}, Duke Math. J. \textbf{55} (1987), no.~4, 919--941.

\bibitem{Zrate}
\bysame, \emph{On the rate of quantum ergodicity. {I}. {U}pper bounds}, Comm.
  Math. Phys. \textbf{160} (1994), no.~1, 81--92.

\end{thebibliography}
\bibliographystyle{amsplain}
\end{document}